\documentclass[12pt]{article}
\usepackage{latexsym,amsmath,amsthm,verbatim,ifthen,amssymb}

\usepackage[all]{xy}

\newcommand{\mQ}{\mathbb{Q}}

\newcommand{\pa}{\left(}
\newcommand{\pb}{\right)}

\newcommand{\cofib}{\rightarrowtail}

\newcommand{\fib}{\twoheadrightarrow}

\newcommand{\quism}{\stackrel{\simeq}{\longrightarrow}}

\newcommand{\sul}[1]{\pa\Lambda #1,d\pb}

\newcommand{\secat}{{\rm secat}}

\newcommand{\msecat}{{\rm msecat}}

\newcommand{\hsecat}{{\rm Hsecat}}

\newcommand{\relcat}{{\rm relcat}}

\newcommand{\cat}{{\rm cat}}
\newcommand{\nil}{{\rm nil}}

\newcommand{\tc}{{\rm TC}}
\newcommand{\mtc}{{\rm mTC}}

\newcommand{\id}{{\rm Id}}

\newcommand{\cdga}{{\bf cdga}}

\newcommand{\ol}{\overline}

\newtheorem{theorem}{Theorem}

\newtheorem{lemma}[theorem]{Lemma}
\newtheorem{proposition}[theorem]{Proposition}

\newtheorem{definition}[theorem]{Definition}

\newtheorem{conjecture}[theorem]{Conjecture}

\begin{document}
\title{The rational sectional category of certain maps\footnotetext{This work was partially supported by FEDER through the Ministerio de Educaci\'on y Ciencia project MTM2013-41768-p and the Belgian Interuniversity Attraction Pole (IAP) within the framework ``Dynamics, Geometry and Statistical Physics'' (DYGEST P7/18)}}
\author{J.G. Carrasquel-Vera\footnotetext{Institut de Recherche en Math\'ematique et Physique, Universit\'e catholique de Louvain, 2 Chemin du Cyclotron, 1348 Louvain-la-Neuve, Belgium.
E-mail: \texttt{jose.carrasquel@uclouvain.be}}}

\date{}

\maketitle
\begin{center}
\emph{To my mother}
\end{center}

\begin{abstract}
We give a simple algebraic characterisation of the sectional category of rational
maps admitting a homotopy retraction. As a particular case we get the
F´elix-Halperin theorem for rational Lusternik-Schnirelmann category and prove the conjecture of Jessup-Murillo-Parent on rational topological complexity. We also give a characterisation for relative category in the sense of Doeraene-El Haouari.
\end{abstract}

 \vspace{0.5cm}
 \vspace{0.2cm}

\maketitle

\section*{Introduction}
\noindent Throughout this work we consider all spaces to be of the homotopy type of simply connected CW-complexes of finite type and use standard rational homotopy techniques which are explained in the excellent text \cite{Bible}. Sectional category is an invariant of the homotopy type of maps introduced by Schwarz in \cite{Schwarz66}. If $f\colon X\rightarrow Y$ is a continuous map, its \emph{sectional category} is the smallest $m$ for which there are $m+1$ local homotopy sections for $f$ whose sources form an open cover of $Y$. Its most studied particular case is the well known Lusternik-Schnirelmann (LS) category of a space $X$ introduced in \cite{Lusternik34} as a lower bound for the number of critical points on any smooth map defined on a smooth manifold $X$. Namely, the LS category of a pointed space $X$, $\cat(X)$, is the sectional category of the base point inclusion map, $*\hookrightarrow X$.\\

A remarkable theorem of F\'elix-Halperin \cite[Theorem 4.7]{Felix82} gives an algebraic characterisation of LS category of rational spaces in terms of their Sullivan models. Explicitly, if $X$ is a space modelled by $\sul{V}$ and $X_0$ is its rationalisation (see \cite{Bible,Su77}) then $\cat(X_0)$ is the smallest $m$ for which the commutative differential graded algebra (cdga) projection \[\rho_m\colon \sul{V}\rightarrow \pa\frac{\Lambda V}{\Lambda^{>m}V},\ol{d}\pb\] admits a homotopy retraction, that is, a strict retraction for a relative Sullivan model for $\rho_m$.\\

Let $f\colon X\rightarrow Y$ be a map such that its rationalisation $f_0$ admits a homotopy retraction $r$. Then, through standard rational homotopy techniques $f$ can be modelled by a retraction $\varphi\colon (B\otimes\Lambda W,D)\rightarrow (B,d)$ of a relative Sullivan algebra $(B,d)\cofib (B\otimes\Lambda W,D)$ modeling $r$. For simplicity in the notation, write $\varphi\colon A\rightarrow B$ and call it from now on an \emph{s-model} of $f$.

\begin{theorem}\label{th:MainIntro}
The sectional category of the rationalisation of $f$, $\secat(f_0)$, is the smallest $m$ for which the cdga\ projection \[A\rightarrow \frac{A}{(\ker\varphi)^{m+1}}\] admits a homotopy retraction.
\end{theorem}

Observe that, choosing $\varphi\colon(\Lambda V,d)\rightarrow \mQ$, this theorem reduces to the F\'elix-Halperin theorem for rational LS-category. On the other hand, it also generalises the Murillo-Jessup-Parent conjecture on rational topological complexity \cite{Jessup12}.\\

Indeed, in his famous paper \cite{Farber03} M. Farber introduced the concept of \emph{topological complexity} of a space $X$, $\tc(X)$, which can be seen as the sectional category of the diagonal map $\Delta\colon X\rightarrow X\times X$. This invariant is used to estimate the \emph{motion planning complexity} of a mechanical system and also has applications to other fields of mathematics \cite{Farber08}. As a direct generalisation of this invariant, Rudyak introduced in \cite{Rudyak10} the concept of higher topological $n$-complexity of a space, $\tc_n(X)$, as the sectional category of the $n$-diagonal map $\Delta_n\colon X\rightarrow X^n$. Several explicit computations of topological complexity of rational spaces have been done in \cite{Carrasquel10,Grant,Jessup12,Lechuga07}. Inspired on the F\'elix-Halperin theorem, Jessup, Murillo and Parent, conjectured that $\tc(X_0)$ is the smallest $m$ such that the projection \[\pa \Lambda V\otimes\Lambda V,d\pb\longrightarrow\pa \frac{\Lambda V\otimes\Lambda V}{K^{m+1}},\ol{d}\pb\] admits a homotopy retraction, where $K$ denotes the kernel of the multiplication morphism $\mu_2\colon \Lambda V\otimes\Lambda V\rightarrow \Lambda V$.\\

Theorem \ref{th:MainIntro} applied to higher topological complexity is a bit more general than the Murillo-Jessup-Parent conjecture. Namely, if $A$ is \emph{any} cdga model for a space $X$, then $\Delta_n$ admits an s-model of the form $\varphi=(\id_A,\eta,\ldots,\eta)\colon A\otimes(\Lambda V)^{\otimes n-1}\rightarrow A$ where $\eta\colon \Lambda V\quism A$ is a Sullivan model for $A$.  From Theorem \ref{th:MainIntro} we immediately deduce:
\begin{theorem}\label{thm:IntroTc}
Let $X$ be a topological space, then $\tc_n(X_0)$ is the smallest $m$ such that the projection \[A\otimes(\Lambda V)^{\otimes n-1}\longrightarrow \frac{A\otimes(\Lambda V)^{\otimes n-1}}{(\ker\ \varphi)^{m+1}}\] admits a homotopy retraction.
\end{theorem}

We remark that, since $\secat(f_0)\le\secat(f)$ \cite{Carrasquel10}, we get algebraic lower bounds for \emph{integral} sectional category which are better than $\nil \ker f^*$. Some of the ideas in this paper come from \cite{Cornea94}.


\section{Fibrewise pointed \cdga s and relative nilpotency}
In this section we develop some technical tools that will be needed later on. Let $\mathcal{C}$ be a J-category in the sense of Doeraene \cite{Doeraene90,Doeraene93} or a closed model category in the sense of Quillen \cite{Quillen67} and fix and object $B$ of $\mathcal{C}$. Consider the \emph{fibrewise pointed category} over $B$ \cite[Pg. 30]{B}, denoted by $\mathcal{C}(B)$, whose objects are factorisations of $\id_B$, $B\stackrel{s_X}{\longrightarrow}X\stackrel{p_X}{\longrightarrow}B$, and whose morphism are morphisms in $\mathcal{C}$, $f\colon X\rightarrow Y$, such that $f\circ s_X=s_Y$ and $p_Y\circ f=p_X$. Such a morphism is said to be a \emph{fibration} ($\fib$), \emph{cofibration} ($\cofib$) or \emph{weak equivalence} ($\quism$) if the underlying morphism $f$ is such in $\mathcal{C}$. With these definitions $\mathcal{C}(B)$ is also either a J-category or a closed model category (note that this structure is not the same as that of \cite{Calcines14b}). We denote by $[X,Y]_B$ the homotopy classes of morphism in $\mathcal{C}(B)$ from the fibrant-cofibrant object $X$ into $Y$.\\

Now, and for the rest of the paper, we particularise on $\mathcal{C}=\cdga$. Remark that the fibrant-cofibrant objects of $\cdga(B)$ are precisely the relative Sullivan algebras $\pa B\otimes\Lambda W, D\pb$ with the natural inclusion $B\hookrightarrow \pa B\otimes\Lambda W, D\pb$ and endowed with a given retraction. In this context, the general property \cite{Quillen67} by which weak equivalences induce bijections on homotopy classes reads:

\begin{lemma}\label{lem:FibrewiseLifting}
Suppose $\theta\colon A\rightarrow C$ is a quasi-isomorphism in $\cdga(B)$ and $\pa B\otimes\Lambda V,D\pb$ a fibrant-cofibrant object of $\cdga(B)$, then composition with $\theta$ induces a bijection $\theta_\#\colon[B\otimes\Lambda V,A]_B\rightarrow [B\otimes\Lambda V,C]_B.$
\end{lemma}

\begin{definition}
Let $A\in\cdga(B)$, its \emph{relative nilpotency index}, $\nil_B(A)$, is the nilpotency index $\nil \ker p_A$ of the ideal $\ker p_A$.  
\end{definition}

The following lemma is crucial. It tells us that we can control the relative nilpotency index of certain homotopy pullbacks of $\cdga(B)$.

\begin{lemma}\label{lem:PBandRelativeNilpo}
Let  $i\colon C\cofib \pa C\otimes\Lambda V, D\pb$ be a cofibration in $\cdga(B)$ such that $D(V)\subset (\ker p_C)\oplus(C\otimes\Lambda^+ V)$ and $p_{C\otimes\Lambda V}(V)=0$. Then, there is an object $N\in\cdga(B)$ weakly equivalent to the homotopy pullback of $i$ and $s_{C\otimes\Lambda V}$ for which $\nil_B\ N=\nil_B\ C+1$. 
\end{lemma}

\begin{proof}In $\cdga(B)$, factor $s_{C\otimes\Lambda V}$ as $B\stackrel{\alpha}{\hookrightarrow}S\stackrel{h}{\fib}C\otimes\Lambda V$ where \[S= B\oplus\pa C\otimes\Lambda V\otimes\Lambda^+(t,dt)\pb,\] in which $t$ has degree $0$, $b(c\otimes v\otimes \xi)=s_C(b)c\otimes v\otimes\xi$, and $h(c\otimes v\otimes t)=c\otimes v$. As $C\otimes\Lambda V\otimes\Lambda^+(t,dt)$ is acyclic, $\alpha$ is a quasi-isomorphism and thus, the homotopy pullback of $i$ and $s_{C\otimes\Lambda V}$ is the pullback
\newdir{ >}{{}*!/-8pt/@{>}}
\[\xymatrix{
M'\ar[d]\ar[r]^{\ol{h}} &C\ar@{ >->}[d]^i\\
S\ar@{->>}[r]_-h &C\otimes\Lambda V\\
}\]
of $i$ and $h$. This is in fact a pullback in $\cdga(B)$ by choosing $p_{M'}=p_C\circ \ol{h}$ and $s_{M'}=(\alpha,s_C)$. To finish, we will construct an object $N$ of $\cdga(B)$ weakly equivalent to $M'$ with $\nil_B\ N=\nil_B\ C+1$.\\

Write $K_\epsilon=\ker \epsilon$ where $\epsilon\colon \Lambda^+(t,dt)\rightarrow\mQ$ is the augmentation sending $t$ to $1$, and consider the $\cdga(B)$ isomorphism \[\eta\colon M\stackrel{\cong}{\longrightarrow}M',\] in which:
\[M= B\oplus\pa C\otimes\Lambda^+(t,dt)\pb\oplus\pa C\otimes\Lambda^+V\otimes K_\epsilon\pb,\]
\[s_M(b)=b,\hspace{0.5cm} p_M(b)=b,\hspace{0.5cm} p_M(c\otimes\xi)=p_C(c)\epsilon(\xi),\hspace{0.5cm} p_M(c\otimes v\otimes \omega)=0,\]
\[\eta(b)=(b,s_C(b)),\hspace{0.5cm}\eta(c\otimes\xi)=(c\otimes1\otimes\xi,c\epsilon(\xi)),\hspace{0.5cm} \eta(c\otimes v\otimes \omega)=(c\otimes v\otimes \omega,0),\] with $b\in B$, $c\in C$, $\xi\in\Lambda^+(t,dt)$, $v\in V$ and $\omega\in K_\epsilon$.\\

Next, write $C=\ker p_C\oplus R$ and consider \[N=B\oplus \pa \ker p_C\otimes\Lambda^+(t,dt)\pb\oplus\pa \ker p_C\otimes\Lambda^+V\otimes K_\epsilon\pb\oplus \pa R\otimes\Lambda^+ V\otimes dt\pb\] which, since $D(V)\subset (\ker p_C)\oplus(C\otimes\Lambda^+V)$, is a sub $\cdga(B)$ of $M$. Moreover, the inclusion $N\hookrightarrow M$ is a weak equivalence in $\cdga(B)$ as the subcomplexes $\pa\ker p_C\otimes \Lambda^+(t,dt)\pb\oplus\pa \ker p_C\otimes\Lambda^+V\otimes K_\epsilon\pb$ and $\pa C\otimes\Lambda^+(t,dt)\pb\oplus\pa \ker p_C\otimes\Lambda^+V\otimes K_\epsilon\pb$ are quasi-isomorphic and the inclusion of quotient complexes $B\oplus\pa R\otimes\Lambda^+V\otimes dt\pb\hookrightarrow B\oplus \pa R\otimes\Lambda^+V\otimes K_\epsilon\pb$ is a quasi-isomorphism.\\

Finally, we have that \[\ker p_N=\pa \ker p_C\otimes\Lambda^+(t,dt)\pb\oplus\pa \ker p_C\otimes\Lambda^+V\otimes K_\epsilon\pb\oplus \pa R\otimes\Lambda^+ V\otimes dt\pb,\]
and thus, a non-trivial product of maximal length in this ideal is given by $z(1\otimes v\otimes dt)$ where $z$ is a non-trivial product of maximal length in $\ker p_C\otimes\Lambda^+(t)$. This proves that $\nil_BN=\nil_BC+1$.
\end{proof}

\section{The main result}\label{sec:relativeNilpAndGaneaModels}
\noindent Let $f\colon X\rightarrow Y$ be a continuous map. Recall from \cite{B,Carrasquel14,Fasso} that, by iterated joins, one can construct an $m$-Ganea map for $f$, $G_m(f)$, fitting into a commutative diagram 
\begin{equation}\label{diag:Ganea}
\xymatrix{
&X\ar[dl]_\iota\ar[dr]^f& \\
P^m(f)\ar[rr]_{G_m(f)}& &Y,\\
}
\end{equation}
and that $\secat(f)\le m$ if and only if $G_m(f)$ admits a homotopy section. Also, if $\varphi\colon A\fib B$ is a surjective model for $f$, then Diagram (\ref{diag:Ganea}) can be modelled by a diagram
\[\xymatrix{
A\ar[dr]_\varphi\ar[rr]^{\kappa_m}&&C_m\ar[dl]^{p_m}\\
&B,&\\
}\]
where $\kappa_m$ models $G_m(f)$ and can be constructed inductively by taking the homotopy pullback of the induced maps by the homotopy pushout of $\varphi$ and any model, $g\colon A\rightarrow D$, of $G_{m-1}(f)$. Standard arguments show that $\secat(f_0)\le m$ if and only if $\kappa_m$ admits a homotopy retraction. One can extend this to:

\begin{definition} Let $f\colon X\rightarrow Y$ be a continuous map. Then:
\begin{enumerate}
\item[(i)] $\msecat(f)\le m$ if and only if $\kappa_m$ admits a homotopy retraction as $A$-module;
\item[(ii)] $\hsecat(f)\le m$ if and only if $\kappa_m$ is homology injective.
\end{enumerate}
\end{definition}
We now give the key model for the $m$-Ganea map $G_m(f)$:
\begin{proposition}\label{pro:GaneanModelNilN}
Let $f$ be a map such that $f_0$ admits a homotopy retraction and let $\varphi\colon A\fib B$ be an s-model for $f$. Then there is a model $\lambda_m$ for $G_m(f)$ which is a morphism in $\cdga(B)$,
\newdir{ >}{{}*!/-8pt/@{>}}
\[\xymatrix{
&B\ar@{ >->}[dl]_s\ar[dr]^{s_m}&\\
A\ar[rr]^{\lambda_m}\ar@{->>}[dr]_\varphi&&C_m\ar@{->>}[dl]^{p_m}\\
&B,\\
}\]
with $\nil_B\ C_m=m$.
\end{proposition}
\begin{proof}
We will proceed by induction. For $m=0$ the assertion holds since $p_0=\id_B$. Suppose $\lambda_{m-1}$ exists. Since $\varphi$ is surjective, one can take a relative Sullivan model for $\varphi$, $\theta\colon \pa A\otimes\Lambda V,D\pb\quism B,$ such that $D(V)\subset(\ker\varphi)\oplus \pa A\otimes\Lambda^+V\pb$ and $\theta(V)=0$. Now take the homotopy pushout

\newdir{ >}{{}*!/-8pt/@{>}}
\[\xymatrix@C=2cm{
A\ar@{ >->}[d]_j\ar[r]^{\lambda_{m-1}}&C_{m-1}\ar@{ >->}[d]^i\ar@/^0.7cm/[ddr]^{p_{m-1}}\\
A\otimes\Lambda V\ar@/_0.7cm/[drr]_\theta\ar[r]_-{\lambda_{n-1}\otimes\id}&C_{m-1}\otimes \Lambda V\ar@{-->}[dr]^{(\theta,p_{m-1})}\\
&&B\\
}\]
and factor $\lambda_{n-1}\otimes\id$ as $q\circ w$ with $q\colon E\fib C_{m-1}\otimes\Lambda V$ a surjective cdga morphism and $w\colon A\otimes\Lambda V\quism E$ a weak equivalence. Then the pullback's universal property
\newdir{ >}{{}*!/-8pt/@{>}}
\[\xymatrix{
A\ar@{-->}[dr]^g\ar@/_0.7cm/[rdd]_{w\circ j}\ar@/^0.7cm/[rdr]^{\lambda_{m-1}}\\
&T \ar[d]\ar[r]^{\ol{q}}&C_{m-1}\ar@{ >->}[d]^i\\
&E\ar@{->>}[r]_-{q}&C_{m-1}\otimes \Lambda V,\\
}\]
gives a model $g$ for $G_m(f)$ which can be seen as a morphism in $\cdga(B)$ by taking $p_T=p_{m-1}\circ\ol{q}$ and $s_T=g\circ s$. Now, define $\beta=i\circ s_{m-1}\colon B\rightarrow C_{m-1}\otimes\Lambda V$ and consider the factorisation of $\beta=q\circ (w \circ j\circ s)$ as a quasi-isomorphism followed by a fibration. On the other hand, consider also the factorisation of $\beta=h\circ \alpha$ as in the proof of Lemma \ref{lem:PBandRelativeNilpo}. Applying \cite[Lemma 1.8]{Doeraene93} to previous factorisations and the following commutative square in \cdga(B),
\newdir{ >}{{}*!/-8pt/@{>}}
\[\xymatrix{
B\ar[d]_{\id_B}\ar[r]^{s_{m-1}}&C_{m-1}\ar@{ >->}[d]^i\\
B\ar[r]_-\beta&C_{m-1}\otimes\Lambda V,\\
}\]
we get quasi-isomorphisms $M'\stackrel{\simeq}{\longleftarrow}\bullet\stackrel{\simeq}{\longrightarrow}T$.
Now, applying Lemma \ref{lem:PBandRelativeNilpo} to $C_{m-1}\cofib C_{m-1}\otimes \Lambda V$, with $s_{C_{m-1}\otimes \Lambda V}=\beta$ and $p_{C_{m-1}\otimes \Lambda V}=(\theta,p_{m-1})$ we get an object $C_m$ of $\cdga(B)$, with $\nil_B\ C_m=m$, which is weakly equivalent $M'$. Observe that we cannot use the pullback's universal property to get a model of $G_m(f)$ because, in general, $\beta\circ \varphi$ does not coincide with $i\circ\lambda_{m-1}$. We get then a diagram in $\cdga(B)$
\newdir{ >}{{}*!/-8pt/@{>}}
\[\xymatrix{
A\ar[r]^g&T&\bullet\ar[l]_-\simeq\ar[r]^-\simeq&M'&\ar[l]_-\simeq C_m.\\
}\]
Since $A$ is a fibrant-cofibrant object of $\cdga(B)$, we can apply Lemma \ref{lem:FibrewiseLifting} to get a model for $G_m(f)$ in $\cdga(B)$, $\lambda_m\colon A\rightarrow C_m$, with $\nil_B\ C_m=m$.
\end{proof}

Given $\varphi\colon A\fib B$ a surjective cdga morphism, consider, for each $m\ge 0$, the cdga projection \[\rho_m\colon A\rightarrow\frac{A}{(\ker \varphi)^{m+1}}.\] Then Theorem \ref{th:MainIntro} is just statement \emph{(i)} in the following:

\begin{theorem}\label{th:Main}
Let $\varphi\colon A\fib B$ be an s-model for a map $f$ such that $f_0$ admits a homotopy retraction. Then: 
\begin{enumerate}
\item[(i)] $\secat(f_0)$ is the smallest $m$ for which $\rho_m$ admits a homotopy retraction;
\item[(ii)] $\msecat(f_0)$ is the smallest $m$ for which $\rho_m$ admits a homotopy retraction as $A$-module;
\item[(iii)] $\hsecat(f)$ is the smallest $m$ such that $H(\rho_m)$ is injective.
\end{enumerate}
\end{theorem}
\begin{proof}
Take from Proposition \ref{pro:GaneanModelNilN} a morphism of $\cdga(B)$, $\lambda_m\colon A\rightarrow C_m$, modelling $G_m(f)$ with $\nil_B\ C_m=m$. Since $\lambda_m((\ker\ \varphi)^{m+1})=0$ we get a commutative diagram
\[\xymatrix{
A\ar[r]^{\lambda_m}\ar[d]_{\rho_m} &C_m\ar[d]^{p_m}\\
\frac{A}{(\ker\ \varphi)^{m+1}}\ar[ur]_-{\overline{\lambda_m}}\ar[r]_-{\ol{\varphi}} & B\\
}\]
and the result follows by standard rational homotopy techniques and \cite[Proposition 12]{Carrasquel10}. 
\end{proof}

Observe that \cite[Example 10]{Carrasquel10} shows that the hypothesis $s$ is a cofibration is necessary.\\

In \cite{St00} D. Stanley gives an example of a map $f$ for which $f_0$ does not admit a homotopy retraction and $\msecat(f)<\secat(f_0)$. Here we state:

\begin{conjecture}
If $f$ is a map and $f_0$ admits a homotopy retraction, then \[\msecat(f)=\secat(f_0).\]
\end{conjecture}

Concerning the rational topological complexity of a given space $X$ and, with the notation in Theorem \ref{thm:IntroTc}, we may define $\mtc(X)$ as the smallest integer $m$ for which the projection \[A\otimes\Lambda V\rightarrow \frac{A\otimes \Lambda V}{(\ker\varphi)^{m+1}}\] admits a homotopy retraction as $A\otimes\Lambda V$-module. Then Theorem \ref{th:Main} (ii) combined with \cite[Theorem 1.6]{Jessup12} gives the Ganea conjecture for $\mtc$.

\begin{theorem}\label{th:GaneaMTC}
Given any space $X$ then $\mtc(X\times S^n)=\mtc(X)+\mtc(S^n).$
\end{theorem}

We finish by presenting, via Theorem \ref{th:Main}, an algebraic description of the rational relative category. Recall \cite{Doeraene13} that the relative category, $\relcat{f}$, of a map $f$ is the smallest $m$ for which $G_m(f)$ of Diagram (\ref{diag:Ganea}) admits a homotopy section $s$ such that $s\circ f\simeq \iota$. Also, in \cite{Doeraene13}, Doeraene and El Haouari proved that $\secat(f)$ and $\relcat(f)$ differ at most by one and conjectured in \cite{Doeraene13b} that they agree on maps admitting a homotopy retraction. Consider then such a map $f$ and $\varphi\colon A\rightarrow B$ and s-model for $f$. This gives a diagram
\newdir{ >}{{}*!/-8pt/@{>}}
\[\xymatrix{
&\pa A\otimes\Lambda Z_{m},D\pb\ar[dr]^{\theta_m}_\simeq&\\
A\ar[dr]_\varphi\ar@{ >->}[ur]^{i_m}\ar[rr]^{\rho_m}&&\frac{A}{(\ker\ \varphi)^{m+1}}\ar[dl]^{\ol{\varphi}}\\
&B,\\
}\]
where $i_m$ is a relative Sullivan model for $\rho_m$.
\begin{theorem}\label{th:MainRelcat}
With previous notation, $\relcat(f_0)$ is the smallest $m$ such that $i_m$ admits a retraction $r$ verifying $\varphi\circ r\simeq \ol{\varphi} \circ \theta_m \mbox{ rel } A$.
\end{theorem}
\begin{proof}
Consider the commutative diagram in the proof of Theorem \ref{th:Main},
where $p_m$ is a model for $\iota$ in Diagram (\ref{diag:Ganea}). Taking $j_m$ a relative model of $\lambda_m$ and applying Lemma \ref{lem:FibrewiseLifting} we get a diagram in $\cdga(B)$
\newdir{ >}{{}*!/-8pt/@{>}}
\[\xymatrix{
&\pa A\otimes \Lambda Z_m,D\pb\ar[dr]^w&\\
A\ar@{ >->}[ur]^{i_m}\ar@{ >->}[rr]_-{j_m}&&\pa A\otimes\Lambda W_m,D\pb.\\
}\]
If $j_m$ admits a retraction $r'$ such that $\varphi\circ r'\simeq p_m\mbox{ rel } A$ then $i_m$ admits a retraction $r:=r'\circ w$ such that $\varphi\circ r=\varphi\circ r'\circ w\simeq p_m\circ \omega=\ol{\varphi}\circ\theta_m\mbox{ rel } A$.
\end{proof}

%

\section*{Acknowledgements}
The author is grateful to the referee for his or her helpful remarks. The author also acknowledges the Belgian Interuniversity Attraction Pole (IAP) for support within the framework ``Dynamics, Geometry and Statistical Physics'' (DYGEST).

\vspace{2cm}
\noindent Institut de Recherche en Math\'ematique et Physique,\\
Universit\'e catholique de Louvain,\\
2 Chemin du Cyclotron,\\
1348 Louvain-la-Neuve, Belgium.\\
E-mail: \texttt{jose.carrasquel@uclouvain.be}

\end{document}